%% file: main.tex
\documentclass[letterpaper, 10 pt, conference]{ieeeconf}
\IEEEoverridecommandlockouts
\usepackage{cite}
\usepackage{amsmath,amssymb,amsfonts}
\allowdisplaybreaks
\usepackage{mathtools,comment}
\usepackage{algorithmic}
\usepackage{graphicx}
\usepackage{tikz}
\usetikzlibrary{%
    decorations.pathreplacing,%
    decorations.pathmorphing%
}
\usetikzlibrary{arrows.meta}
\usetikzlibrary{automata,positioning}
\usepackage{textcomp}
\usepackage{xcolor}
\let\originalleft\left
\let\originalright\right
\renewcommand{\left}{\mathopen{}\mathclose\bgroup\originalleft}
\renewcommand{\right}{\aftergroup\egroup\originalright}
\newcommand{\bE}{\mathbb{E}}

\newcommand{\cN}{\mathcal{N}}

\newcommand{\ust}{^{\star}}

\newcommand{\bR}{\mathbb{R}}

\newcommand{\cI}{\mathcal{I}}

\newcommand{\bP}{\mathbb{P}}

\newcommand{\cT}{\mathcal{T}}
\newcommand{\cB}{\mathcal{B}}

\newcommand{\cL}{\mathcal{L}}

\newcommand{\D}{\Delta}

\newcommand{\Te}{\Tilde{e}}
\newcommand{\TD}{\Tilde{\Delta}}
\newcommand{\Td}{\Tilde{d}}
\newcommand{\Tc}{\Tilde{c}}
\newcommand{\Tph}{\Tilde{\phi}}
\newcommand{\Tu}{\Tilde{u}}
\newcommand{\TV}{\Tilde{V}}
\newcommand{\TQ}{\Tilde{Q}}
\newcommand{\Tp}{\Tilde{p}}

\newcommand{\itg}{\int_{\bR_+}}

\newcommand{\vph}{\varphi(\TD_+,a\TD)}
\newcommand{\cQ}{(1-c) e^{\g(\lambda + \D^2)}}

\newcommand{\gl}{ce^{\g\lambda}}

\newcommand{\gd}{e^{\g\D^2}}
\newcommand{\gTd}{e^{\g\TD^2}}
\newcommand{\sumc}{\sum_{c_+ \in \{0,1\}}}
\newcommand{\sumTc}{\sum_{\Tc_+ \in \{0,1\}}}

\newcommand{\g}{\gamma}
\DeclareMathOperator*{\argmin}{arg\,min}

\newcommand{\nal}[1]{\begin{align*}#1\end{align*}}
\newcommand{\al}[1]{\begin{align}#1\end{align}}

\newtheorem{definition}{Definition}%[section]
\newtheorem{theorem}{Theorem}%[section]
\newtheorem{proposition}{Proposition}%[section]
\newtheorem{corollary}{Corollary}%[theorem]
\newtheorem{lemma}{Lemma}

\newif\ifuseRomanappendices
\useRomanappendicesfalse

\def\BibTeX{{\rm B\kern-.05em{\sc i\kern-.025em b}\kern-.08em
    T\kern-.1667em\lower.7ex\hbox{E}\kern-.125emX}}

\title{\LARGE \bf
 Optimal Risk-Sensitive Scheduling Policies for Remote Estimation of Autoregressive Markov Processes
}

\author{Manali Dutta and Rahul Singh% <-this % stops a space
\thanks{The authors are with the Department of Electrical and Communication Engineering,
Indian Institute of Science, Bengaluru, Karnataka 560012, India (e-mail: manalidutta@iisc.ac.in and rahulsingh@iisc.ac.in)}%
}
\begin{document}

\maketitle

\begin{abstract}
We design scheduling policies that minimize a risk-sensitive cost criterion for a remote estimation setup.
~Since risk-sensitive cost objective takes into account not just the mean value of the cost, but also higher order moments of its probability distribution, the resulting policy is robust to changes in the underlying system's parameters.~The setup consists of a sensor that observes a discrete-time autoregressive Markov process, and at each time $t$ decides whether or not to transmit its observations to a remote estimator using an unreliable wireless communication channel after encoding these observations into data packets.~We model the communication channel as a Gilbert-Elliott channel~\cite{laourine2010betting,chakravorty2019remote,10384144} to take into account the temporal correlations in its fading. Sensor probes the channel~\cite{laourine2010betting} and hence knows the channel state at each time $t$ before making scheduling decision.~The scheduler has to minimize the expected value of the exponential of the finite horizon cumulative cost that is sum of the following two quantities (i) the cumulative transmission power consumed, (ii) the cumulative squared estimator error. We pose this dynamic optimization problem as a Markov decision process (MDP), in which the system state at time $t$ is composed of (i) the instantaneous error $\Delta(t):= x(t)-a\hat{x}(t-1)$, where $x(t),\hat{x}(t-1)$ are the system state and the estimate at time $t,t-1$ respectively, and (ii) the channel state $c(t)$. We show that there exists an optimal policy that has a threshold structure, i.e., at each time $t$, for each possible channel state $c$, there is a threshold $\D\ust(c)$ such that if the current channel state is $c$, then it transmits only when the error $\D(t)$ exceeds $\D\ust(c)$. Our analysis proceeds by constructing a certain ``folded MDP''~\cite{chakravorty2018sufficient} that is much more amenable to analysis than the original MDP. We show structural results for this folded MDP, and finally unfold this to obtain the structural result for the original MDP.
\end{abstract}

\input{introduction}
\input{problem_formulation}
\input{MDP_formulation}
\input{Structural_results}

\input{conclusions}
\input{appendix}

\bibliographystyle{IEEEtran}
\bibliography{refs}

\end{document}

%% file: introduction.tex
\section{Introduction}
We design risk-sensitive~\cite{jacobson1973optimal,whittle1990risk} optimal cheduling policies which solve the problem faced by a sensor in a remote state estimation setup~\cite{ren2017infinite,chakravorty2019remote,dutta2023optimal}. More specifically, the networked control system (NCS) of interest consists of an autoregressive Markov process that is observed by a sensor. Sensor encodes these observations into data packets, and then transmits them to a remote estimator over an unreliable wireless communication channel.~This remote estimator is spatially distributed from the source and the sensor, and estimates the state of the underlying source process. Packet transmission attempts consume energy.

The problem faced by the sensor while making scheduling decisions is described as follows. If it continually transmits packets, then this ensures high quality estimates of the process at the remote estimator. However, this strategy is not energy efficient since packet transmissions consume energy. On the other hand, if it does not transmit packets for long time durations in order to save energy, then the quality of the estimates is degraded.~In order to strike a balance between minimizing the estimation error, and keeping the power consumption low, the sensor implements scheduling policies that make transmission decisions dynamically, based on the information currently available with it.

Our focus in this work is on minimizing the following risk-sensitive exponential cost criteria over a finite time horizon $\bE \exp \gamma \left( \sum_{t=0}^{T-1} (x(t)-\hat{x}(t))^2 + \lambda u(t) \right) $, where $x(t),\hat{x}(t)\in\bR$ are the process state and its estimate at time $t$ respectively, while $u(t)$ is the decision variable which indicates whether ($u(t)=1$) or not ($u(t)=0$) a packet is attempted for transmission at time $t$. $\lambda>0$ is the unit price per transmission, while $\gamma>0$ is the risk-sensitivity parameter~\cite{bacsar2021robust}. Next, we give a brief overview of risk-sensitive control, and also discuss the utility of scheduling policies that minimize such an objective. 
\subsection{Risk-Sensitive Scheduling}
Consider a dynamical system that operates for $T$ steps. If $d(t)$ is the cost incurred by it at time $t$, then the corresponding \textit{risk-sensitive cost} is given as follows,
\al{\label{rs-obj}
\frac{1}{\gamma} \ln {\bE[e^{\gamma\sum_{t=0}^{T-1} d(t)}] }, 
}
where $\gamma > 0$ is called the risk-sensitivity parameter, and expectation is taken with respect to the underlying probability measure. As compared with the corresponding \textit{risk-neutral} cost $\bE\left[\sum_{t=0}^{T-1} d(t)\right]$, we note that while the risk-neutral cost penalizes only the mean of the cumulative cost, the risk-sensitive cost also takes into consideration all the higher order moments of the cost. Indeed, the Taylor series expansion for~\eqref{rs-obj} for small values of $\gamma$ around $0$, can be approximated as follows~\cite{bacsar2021robust},
\al{
%& 	\bE[e^{\gamma\sum_{t=0}^{T-1} d(t)}] \approx 
\bE\left[ \sum_{t=0}^{T-1} d(t)\right] + \frac{\gamma}{2} Var\left[\sum_{t=0}^{T-1} d(t)\right]  + O(\gamma^2),\label{tayexp}
}
where for a random variable $X$, $Var[X]$ represents its variance.%Thus, in this case 
%it aims to minimize a cost that assigns a positive weight to the variance of the total cost incurred, as opposed to risk-neutral objective where only the mean cost is considered~\eqref{rn-obj}.  
~Hence, a control policy that optimizes this risk-sensitive cost, is robust to variations in the system parameters, possibly induced by an adversary~\cite{bacsar2021robust}. Consequently, a scheduling policy that makes decisions regarding packet transmissions by optimizing the risk-sensitive criteria, is averse to uncertainties and undesirable variations in the system. 

Risk-sensitive cost optimization setup is more general than the classical risk-neutral optimization. Indeed, we note from~\eqref{tayexp} that in the limit $\gamma \rightarrow 0 $, the objective~\eqref{rs-obj} reduces to the risk-neutral cost $\bE\left[\sum_{t=0}^{T-1} d(t)\right]$, so that we recover the classical risk-neutral stochastic controls from the risk-averse formulation~\cite{whittle1990risk}. Robustness of risk-sensitive controls, and its connection with the robust control \slash $H_\infty$ control~\cite{francis1987linear} are well-known by now~\cite{whittle1990risk},~\cite{dupuis2000robust},~\cite{whittle2002risk}. The goal of a robust controller is to deal with model uncertainties.~\cite{glover1988state} is one of the first works that establish a link between risk-sensitive control and robust control.~Subsequently, extensive research efforts have been directed towards finding connections between these two fields~\cite{fleming1995risk, khalil1996robust, fleming1997risk, fleming1999risk}. Additionally, it has been shown that as $\gamma \rightarrow \infty$, the risk-sensitive objective approaches the minimax objective~\cite{coraluppi1999risk}.~In a minimax optimization problem, the quality of a solution is judged by its performance in the worst possible scenario.~Thus, the connection between the risk-sensitive objective and the minimax objective suggests that a controller obtained by optimizing the risk-sensitive cost with the risk-sensitivity parameter set at a high value, is risk-averse, and hence exhibits a higher tolerance for uncertainties in the system as compared with a risk neutral optimal controller.

In summary, we are motivated to design policies for NCS by optimizing risk-sensitive objective since 
it takes into account not just the mean value of the cost, but also its higher order moments. This will ensure that the system is robust to unpredictable changes. This is important since a major concern for NCS is their susceptibility to cyber attacks~\cite{cardenas2009challenges},~\cite{fawzi2014secure}.~This arises mainly due to their openness to the digital world which poses significant security challenges. For example, cyber attacks may lead to packet losses in a wireless communication channel~\cite{shu2014privacy},~\cite{cetinkaya2016networked}, false data injection~\cite{li2018false},~\cite{sargolzaei2019detection}, and introduction of delays into signals used in NCS~\cite{victorio2021secure},~\cite{sargolzaei2022secure}. Hence, to protect the network against such malicious attacks, the risk-sensitive cost criterion serves as a beneficial framework~\cite{befekadu2011risk,befekadu2015risk,singh2017risk}.~Despite this, there has been a limited work on designing such policies for NCS.  %In this work, we consider the problem of designing a scheduling policy that minimizes the expected value of the exponential of the cumulative cost composed of transmission energy consumed and estimation error, i.e. we derive optimal risk-sensitive scheduling policies. Such a policy is more robust to adversaries. 
We now discuss prior works on risk-sensitive control and remote estimation problem. Remote state estimation problem is a central topic in the field of NCS~\cite{hespanha2007survey}.
\subsection{Literature Review}
\emph{Risk-Sensitive control of MDPs}: The study of risk-sensitive control of Markov Decision Processes (MDPs) was initiated in~\cite{howard1972risk}. It studied discrete time MDPs that have finite state and action spaces. Since then, there have been numerous works that address various aspects of risk-sensitive control for various types of processes. Further details of these works can be found in~\cite{biswas2023ergodic},~\cite{bauerle2023markov}. 

\emph{Risk-Sensitive control of linear systems}: 
The work on risk-sensitive control of Linear Quadratic Gaussian (LQG) systems~\cite{kumar2015stochastic} was initiated in~\cite{jacobson1973optimal}. It considers linear systems driven by white Gaussian noise in which the performance cost is quadratic in system state and controls.~An important finding is that unlike the risk-neutral control problem, the optimal controller is now also a function of the variance of the Gaussian noise. Since then, several works have studied various aspects of risk-sensitive controls for LQG systems, more details can be found in~\cite{whittle1990risk}.

\emph{Optimal policies for remote estimation in NCS}: 
We now describe works that address various issues faced while optimizing the performance in a remote estimation setup.~Consider a process which is modeled as a linear system driven by Gaussian noise, and an estimator that is located at a different location is tasked with generating its estimates. Packet transmissions consume energy, and there is a sensor that has to dynamically choose when to transmit packets.~The design problems associated with such a setup can be broadly categorized into the following three types: (i) optimizing the estimator for a given scheduler, (ii) for a given estimator, optimizing the scheduling decisions regarding when to transmit packets, and (iii) designing jointly optimal scheduler and estimator. We now discuss works that solve problems (i), (ii), and (iii) in the context of both risk-neutral and risk-sensitive objective. We firstly describe works that study (i)-(iii) for the classical risk-neutral objective, which is then followed by a discussion on its risk-sensitive counterpart. 

\textit{Risk-neutral objective:} For problem (i), Kalman filter~\cite{kalman1960new} serves as the backbone for deriving optimal estimator or minimizing the mean square error in the risk-neutral case.~\cite{sinopoli2004kalman} shows that Kalman filter is optimal when there are intermittent observations due to packet losses suffered while communicating packets from sensor to estimator over wireless networks.~The work~\cite{han2015stochastic} considers a vector source process, and derives optimal estimators for two different classes of scheduling policies, both of which are of ``threshold-type.'' The first class of policies transmit only when a function of the current state observation exceeds a certain threshold, while the second class of policies transmit only when a function of the current measurement innovation, i.e. the difference between the current measurement and its \textit{a priori} estimate, exceeds a certain threshold.~It is then shown that in both the cases, the optimal remote estimator satisfies Kalman update equations, with a modified Kalman gain. Several variants of Kalman filter have been proposed as optimal estimators in order to compensate for delays and packet losses occurring in wireless communication networks~\cite{schenato2007foundations,schenato2008optimal,li2016wireless}. %Very often the optimal estimator turns out to be Kalman-like \rs{what does this mean?}.

We now discuss works that address the issue (ii) mentioned above for risk-neutral objective. The works~\cite{chakravorty2018sufficient,leong2018transmission,wu2019learning,dutta2023optimal} solve (ii) under a broad range of assumptions on the wireless communication channel. The estimator is Kalman-like, i.e., the estimator updates its plant state estimate with the received update upon successful delivery of a packet from the sensor, otherwise it estimates the plant state based on the current information available to it.  It is then shown that there exists an optimal scheduling policy that has a threshold structure.~\cite{chakravorty2018sufficient} allows the transmitter to transmit at various power levels. Packet losses are i.i.d., and the packet loss probability is a known function of the transmission power.~It is then shown that there exists an optimal scheduling policy that has a threshold structure with respect to the current error, i.e., the difference between the current state value and its \textit{a priori} estimate.~\cite{leong2018transmission} assumes i.i.d. packet losses with known loss probability, and then derives optimal scheduling policy when the sensor has constraints on its average energy consumption. Optimal policy is shown to have a threshold structure with respect to the variance of the difference between the current state of the process, and its estimate.~\cite{wu2019learning} also considers an i.i.d. loss model, but assumes that the packet loss probability is unknown. It shows that the optimal policy has a threshold structure with respect to the time elapsed since the last successful transmission.~\cite{dutta2023optimal} models the wireless communication channel as a Gilbert-Elliott channel~\cite{laourine2010betting}, and assumes that the channel state is not known to the sensor. It shows that there exists an optimal scheduling policy that exhibits a threshold structure with respect to the current belief state of the channel state, i.e., the sensor transmits only when the conditional probability that the channel is good, exceeds a certain threshold which is a function of the current value of the error.~Several works~\cite{chakravorty2016remote,ren2017infinite,chakravorty2019remote} consider the problem of designing jointly optimal estimator and scheduler, i.e., the problem (iii) stated above. It is shown in these works that under various assumptions on the channel model, there exists a policy that has a threshold structure with respect to the error, and a Kalman-like estimator, that are jointly optimal.~\cite{chakravorty2016remote} assumes that the packet losses in the wireless channel are i.i.d across times. Both~\cite{ren2017infinite} and~\cite{chakravorty2019remote} model the state of the wireless channel as a Markov process. While~\cite{ren2017infinite} assumes that the sensor knows the channel state instantaneously,~\cite{chakravorty2019remote} assumes its knowledge with a delay of one unit.

\textit{Risk-sensitive objective:} The pioneering work~\cite{speyer1992optimal} considers the problem of designing an estimator that minimizes the risk-sensitive cost associated with the cumulative estimation error, and shows that when there are continual transmissions of observations without any packet losses, then the optimal estimator is a linear filter.\cite{huang2019robust} fixes the scheduling policy to be of threshold-type with respect to a function of the current value of the sensor's measurement of the source process, and shows that the optimal estimator has a ``Kalman-like'' structure, i.e., the aprior and posterior state estimates evolve in a recursive manner similar to the Kalman filter, but with a modified gain, and coefficients that depend upon the risk-sensitivity parameter. 
To the best of our knowledge, there are no existing works that explore the design of an optimal scheduling policy for the sensor, or jointly optimal transmission policy for the sensor and estimator in the context of risk-sensitive cost.
\subsection{Contributions}
The current work designs risk-sensitive scheduling policies for a remote estimator in which a sensor transmits observations of a discrete-time autoregressive (AR) process over a fading wireless channel that is modeled as a Gilbert-Elliott channel~\cite{laourine2010betting,chakravorty2019remote}. This type of channel model is more realistic as compared to an i.i.d. packet drop model~\cite{laourine2010betting}, since it is able to describe the temporal correlations in wireless channel properties. Gilbert-Elliott channel can also be used to model burst-noise channels, where multiple consecutive packets may be lost due to channel fading or interference~\cite{gilbert1960capacity}. The system operating cost considered is the sum of the cumulative transmission power, and the estimation error incurred over a finite horizon.

As is discussed next, minimizing the risk-sensitive objective is much more challenging than the risk-neutral case. We list two major challenges:

C1) A popular approach to solve risk-neutral infinite horizon undiscounted MDPs is the vanishing discount approach~\cite{hernandez2012discrete}. One considers a sequence of discounted MDPs with discount factor converging to $1$, and recovers an optimal policy for the undiscounted problem in the limit the discount factor approaches unity. The success of this approach hinges on the fact that the discounted risk-neutral MDPs admit an optimal stationary policy. However, infinite horizon discounted risk-sensitive MDPs, in general, might not admit a stationary optimal policy~\cite{rojas1999controlled},~\cite{di1999risk}.~This is in sharp contrast with the case of risk-neutral MDPs~\cite{hernandez2012discrete}. Consequently, one cannot employ the vanishing discount approach, that has been used extensively in order to solve the risk-neutral average cost criteria, in order to solve the risk-sensitive MDPs~\cite{rojas1999controlled}.

C2) Since the risk-sensitive cost criterion is multiplicative in nature, the cost at the current time is a function of the history till that time~\cite{speyer1992optimal}. As a result, the linearity property of expectation which can be easily used in additive cost, cannot be directly applied in the risk-sensitive criteria as is shown in~\cite{speyer1992optimal} which considers the problem of deriving an optimal estimator. This makes the analysis more difficult since now we have to consider the entire history leading up to the current time.%\rs{fine, but then how does this affect the analysis? Does it lead to some unexpected behavior? }

Our contributions are as follows:
%For this purpose, we pose the sensor scheduling problem as a dynamic optimization problem.

(1) To the best of our knowledge, ours is the first work to study the problem of designing risk-sensitive scheduling policy for a remote state estimation setup. As an initial attempt to address C1), we consider minimizing the expected value of the exponential of the cumulative cost incurred during a finite time horizon that is a weighted sum of the cumulative transmission power, and the cumulative squared estimation error. We pose this dynamic optimization problem as a MDP in Section~\ref{sec:MDP formulation}, in which the system state comprises of the error $x(t)-a\hat{x}(t-1)$, and the current state of the wireless channel. 

(2) In contrast to the risk-neutral case~\cite{hernandez2012discrete} where the Bellman equation is additive, in the risk-sensitive cost it is multiplicative~\cite{bauerle2014more}. We show in Section~\ref{val-iter} that our model satisfies certain technical assumptions~\cite{bauerle2014more}, and hence we can use the value iteration algorithm to solve the MDP. Moreover, we show the existence of an optimal deterministic Markov policy, i.e., it makes decisions only on the basis of the current state and time. This addresses C2). This is because, at the current time step, it now suffices to store only the previous time step state information and ignore the history. This also reduces the computational complexity and memory requirements on the policy.

(3) The analysis of the MDP is complicated by the fact that the error term $\Delta(t)$~\eqref{delta}, which is part of the system state, assumes both negative and non-negative values. We instead analyze a certain ``folded MDP'' which was introduced in~\cite{chakravorty2018sufficient}, and this significantly simplifies the analysis since in the folded MDP the error assumes only non-negative values. %However, the work in~\cite{chakravorty2018sufficient} considers a risk-neutral objective.

(4) In Section~\ref{sec:structural results}, we establish a novel structural result for the optimal scheduling policy that minimizes the risk-sensitive cost criterion. Specifically, we show that there exists an optimal scheduling policy that exhibits a threshold structure with respect to the error, i.e., for each value of the channel state $c$, there exists a threshold such that the sensor transmits only when the magnitude of the current error exceeds this threshold. Such a structure reduces the policy search space and is easy to implement.

\textit{Notation}: Let $\bR, \bR_+, \bR_-$ denote the set of real numbers, non-negative and negative real numbers, respectively. $\bP(\cdot)$, $\bE(\cdot)$ denote the probability of an event and expectation of a random variable respectively.~$\mathcal{N}(\mu,\sigma^2)$ denotes the Gaussian distribution with mean $\mu$ and variance $\sigma^2$, and $\delta_x(\cdot)$ denotes the delta function with unit mass at $x$.

%% file: problem_formulation.tex
\section{Problem Formulation\label{problem_formulation}}
We introduce the remote state estimation setup in Section~\ref{system-model}, and then formulate the optimal scheduling problem based on a risk-sensitive cost criterion in Section~\ref{risk-sensitive cost}.

\subsection{System Model\label{system-model}}
Consider a remote state estimation setup as shown in Fig.~\ref{fig:setup} that consists of a sensor which observes a discrete-time AR Markov process $\{x(t)\}_{t =0}^{T}$. The state of the process evolves as follows,
\al{\label{source}
x(t+1)=ax(t) +w(t),~t=0,1,2,\ldots, T-1,
}
where the initial state is $x(0) \sim \cN(0,1)$, $a, x(t) \in \bR$, and $w(t)$ is an i.i.d. Gaussian noise process that satisfies $w(t) \sim \cN(0,\sigma^2)$. The sensor encodes its observations into data packets before transmitting them to the remote estimator. At each time $t \in \{0,1,\ldots,T\}$, the sensor has to decide on whether ($u(t)=1$) or not ($u(t)=0$) to attempt a packet transmission. We assume that each transmission attempt incurs $\lambda$ units of energy, where $\lambda > 0$. Packets are transmitted over an unreliable wireless communication channel.~The state of the channel at time $t$ is denoted by $c(t) \in \{0,1\}$.~$c(t)=0$ represents that the channel is in bad state at time $t$, and hence any transmission attempt at time $t$ by the sensor is unsuccessful. $c(t)=1$ denotes that the channel is in a good state, so that any packet which is attempted at time $t$ is successfully delivered to the remote estimator. We model the channel state process $\{c(t)\}_{t =0}^{T}$ as a two-state Markov process. This is popularly known as the Gilbert-Elliott channel~\cite{chakravorty2019remote,10384144}. Such a channel has memory, and can be used to model the temporal correlations in a wireless channel, in contrast to a channel modeled with i.i.d. packet drops. This allows for a more realistic representation of the wireless channel. We let $p_{01}$ be the probability with which channel state at next time step is $1$ given that currently it is in state $0$, and similarly let $p_{10}$ be the probability with which it is $0$ at next step, when currently it is in state $1$.
%The parameters of this Markov process are as follows,
%\al{
%p_{01} := \bP(c(t+1)=1 \mid c(t)=0),\label{p_01}\\
%p_{11} := \bP(c(t+1)=1 \mid c(t)=1).\label{p_11}
%}

\begin{figure}[t]
    %\centering
    \begin{tikzpicture}[scale=0.75]
        % "Auto-regressive process"
        \draw (0, 0) rectangle (2, 1.2);
        \node[align=center, font=\footnotesize] at (1, 0.6) {AR process};
        \draw[->] (2, 0.6) -- (3, 0.6) node[midway, above, font=\footnotesize]{$x(t)$};
        %"Sensor"
        \draw (3, 0) rectangle (4.3, 1.2);
        \node[align=center, font=\footnotesize] at (3.65, 0.6) {Sensor};
        
        \draw[-] (4.3,0.6) -- (4.8,0.6);
        \draw[-] (4.8,0.6) -- (5.2,0.9);
        \draw[->] (5.05, 0.6) -- (5.8, 0.6);
        %"action"
        \draw[dashed, ->,dash pattern=on 2pt off 2pt] (4.9,0.85) arc (90:-40:0.3);
        \node[font=\small] at (5.4, 0.17) {$u(t)$};
        %"G-E channel"
        \draw (5.8, 0) rectangle (8, 1.2);
        \node[align=center, font=\footnotesize] at (6.9, 0.6) {Gilbert-Elliott \\ channel $c(t)$};
        %"Estimator"
        \draw[->] (8, 0.6) -- (8.9, 0.6);
        \draw (8.9, 0) rectangle (10.5, 1.2);
        \node[align=center, font=\footnotesize] at (9.7, 0.6) {Estimator};
        \draw[->] (10.5, 0.6) -- (11.3, 0.6) node[midway, above, font=\footnotesize]{$\hat{x}(t)$};
    
    \begin{scope}[->,>=stealth,shorten >=0.5pt,node distance=1.2cm,on grid,auto,state/.style={circle,draw,minimum size=0.2cm}]
        % Nodes
        \node[state,font=\scriptsize] (0) at (6.1,2.1) {0};
        \node[state, right=of 0,font=\scriptsize] (1) {1};
        
        % Arrows
        \path (0) edge [bend left] node[font=\scriptsize] {$p_{01}$} (1)
              (1) edge [bend left] node[font=\scriptsize] {$p_{10}$} (0);
        \draw[->] (0) to [out=215, in=140, looseness=4] node[pos=0.75, above,font=\scriptsize] {$1 - p_{01}$} (0);
        \draw[->] (1) to [out=45, in=-35, looseness=4] node[pos=0.3, above,font=\scriptsize] {$1-p_{10}$} (1);
        \end{scope}
    \end{tikzpicture}
    \caption{Remote state estimation setup. Source process evolves as $x(t+1) = a x(t) + w(t)$, where $x(t),a \in \bR$, the noise $w(t) \sim \cN(0,1)$, decision variable $u(t) \in \{0,1\}$, channel state $c(t) \in \{0,1\}$, and estimator state $\hat{x}(t) \in \bR$.}
    \label{fig:setup}\vspace{-1em}
\end{figure}
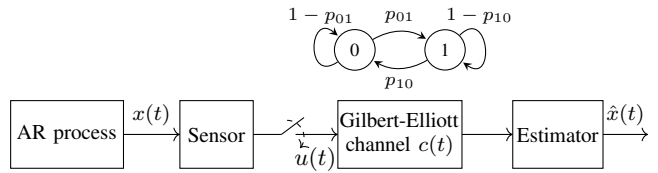
We assume that the channel state is known instantaneously at the sensir.~This is possible since the sensor probes the channel, for example, by sending a probing packet at each time $t$~\cite{laourine2010betting}.% We assume that probing the channel incurs no cost, and . %As a result, if the channel state is found to be bad $(c(t)=0)$, then the sensor does not transmit any packet, i.e., $u(t)=0$. Otherwise, if the channel is in a good state $(c(t)=0)$, only then the sensor decides whether to transmit $(u(t)=1)$ or not $(u(t)=0)$.
%already mentioned once-->Moreover, if the sensor decides to transmit a packet at time $t$ when the channel is in a good state $(c(t)=1)$, then it will be successfully received by the estimator. 
~Let $\hat{x}(t)$ be the state of the estimator at time $t$. The estimate $\hat{x}(t)$ evolves as follows, for $t=1,2,\ldots,T$, we have,
\al{\label{x-hat}
\hat{x}(t) =
\begin{cases}
    a\hat{x}(t-1) &\mbox{ if } u(t)c(t)=0,\\
    x(t) &\mbox{ if } u(t)c(t) =1,
\end{cases}
}
where $\hat{x}(0) =0$. The information available with the sensor at time $t$ is given by,
\al{\label{info-t}
\cI(t):= \Big(\{x(s),c(s)\}_{s=0}^{t},\{u(s)\}_{s=0}^{t-1}\Big).
}
The scheduler at the sensor makes decision $u(t)$ at time $t$ as a function of the information available to it till time $t$, i.e.,
\al{\label{u(t)}
u(t) = \phi_t(\cI(t)),
}
where $\phi_t: \cI(t) \rightarrow u(t)$ is a measurable function, and $u(t) \in \{0,1\}$. The collection $\phi := (\phi_0, \phi_1, \ldots, \phi_{T})$ is a scheduling policy that makes decisions regarding packet transmissions. 
\subsection{Risk-Sensitive Cost\label{risk-sensitive cost}}
Define the following cost function,
\al{\label{instan-cost}
g(x, \hat{x}, u) := \lambda u + (x-\hat{x})^2.
}
%\rs{Note $g(x,\hat{x},u)$, changed ordering..}
Then, the instantaneous cost incurred at time $t$ is $g(x(t),\hat{x}(t),u(t))$.~It is sum of two terms (i) transmission energy $\lambda \cdot u$, and (ii) the squared estimation error $(x-\hat{x})^2$.

We are interested in solving the following finite-horizon risk-sensitive dynamic optimization problem~\cite{bauerle2014more} for the model described in Section~\ref{system-model},
\al{
	\label{obj-fin-hor}
\min_{\phi} \frac{1}{\gamma} \log\bE_{\phi}[e^{\gamma \sum_{t=0}^{T} g(x(t),\hat{x}(t),u(t))}],
}
where $\phi$ is a scheduling policy, $\gamma > 0$ is the risk-sensitivity parameter, $\bE_{\phi}$ denotes the expectation taken w.r.t. the measure induced by policy $\phi$, and $g(x,\hat{x},u)$ is given by~\eqref{instan-cost}.~Since $\log$ is a strictly increasing function,~\eqref{obj-fin-hor} can equivalently be stated as follows,
\al{\label{obj-fin-hor-w/o-log}
%	\textit{Problem 1:}
\min_{\phi} \bE_{\phi}[e^{\gamma \sum_{t=0}^{T} g(x(t),\hat{x}(t),u(t))}], %\tag{Problem 1}
}
where the cost function $g$ is as in~\eqref{instan-cost}. 

%% file: MDP_formulation.tex
\section{MDP formulation}\label{sec:MDP formulation}
We now formulate the problem~\eqref{obj-fin-hor-w/o-log} as a MDP.~Section~\ref{val-iter} discusses how to use the value iteration algorithm to solve this MDP.~Section~\ref{folded MDP} then constructs a certain ``folded MDP'' to simplify its analysis.

Consider the following error process $\{\Delta(t)\}_{t = 0}^{T}$,
\al{\Delta(t):=x(t)-a\hat{x}(t-1),\label{delta}}
where we let $\Delta(0) = 0$. From~\eqref{x-hat} we have that the evolution of $\{\Delta(t)\}$ is given as follows,
\al{\label{delta-recur}
\Delta(t+1) =
\begin{cases}
    a\Delta(t)+w(t) &\mbox{ if } u(t)c(t) = 0,\\
    w(t) &\mbox{ if } u(t)c(t) = 1.
\end{cases}
}
After performing some algebraic manipulations, we have that the instantaneous cost~\eqref{instan-cost} can equivalently be written in terms of $(\Delta,c,u)$ instead of $(x,\hat{x},u)$ as follows, %now be re-written as, 
\al{\label{new-instan-cost}
d(\D,c,u) := \lambda u + (1-uc)\D^2.
}
Instead of solving~\eqref{obj-fin-hor-w/o-log}, we now consider the following equivalent problem,
\al{
	\label{obj-fin-hor-w/o-log_1}
	\min_{\phi} \bE_{\phi}[e^{\gamma \sum_{t=0}^{T} d(\Delta(t),c(t),u(t))}],
	}
where $\gamma > 0$ is the risk-sensitivity parameter, $\bE_{\phi}$ denotes the expectation taken w.r.t. the measure induced by policy $\phi$, and $\D(t)$ is given by~\eqref{delta}.~We now show that~\eqref{obj-fin-hor-w/o-log_1} can be formulated as a MDP in which the state at time $t$ is given by $(\D(t),c(t))$ and control $u(t) \in \{0,1\}$.
\begin{lemma}
\label{lemma:u=f(del,c)}
    For the purpose of solving~\eqref{obj-fin-hor-w/o-log_1}, there is no loss of optimality in restricting the class of scheduling policies in~\eqref{u(t)} to those which have the following form,
    \al{
    u(t) = \phi_t(\Delta(t),c(t)).\label{u=f(del,c)}
    }
\end{lemma}
\begin{proof}
    The proof proceeds by showing that the process $\{\D(t),c(t)\}_{t =0}^{T}$ is a Markov Decision Process (MDP) with control $u(t)$. For this, we will show that $(\D(t),c(t))$ is an information state~\cite{kumar2015stochastic} at the sensor, i.e., for $\cI(t)$, $d(\D,c,u)$ given by~\eqref{info-t} and~\eqref{new-instan-cost} respectively,
    \al{
    \text{(i)} \quad &\bP(\D(t+1),c(t+1) \mid \cI(t), u(t))\notag\\
    & =\bP(\Delta(t+1), c(t+1) \mid \D(t),c(t),u(t)),\notag\\
    \text{(ii)} \quad &\bE[d(\D(t),c(t),u(t)) \mid \cI(t),u(t)] \notag\\
    & = \bE[d(\D(t),c(t), u(t)) \mid \D(t),c(t),u(t)].\notag
    }
    First, we consider (i).
    \al{
    & \bP(\Delta(t+1),c(t+1) \mid \cI(t), u(t))\notag\\
    & = \bP(\D(t+1) \mid \cI(t),c(t+1) , u(t))\bP(c(t+1) \mid \cI(t), u(t)) \notag\\
    & = \bP(\D(t+1) \mid \D(t), u(t)) \bP(c(t+1) \mid c(t))\notag\\
    & =\bP(\Delta(t+1),c(t+1) \mid \Delta(t),c(t), u(t)),\notag
    }
    where second equality follows from~\eqref{delta}, and the Markovian nature of $\D$~\eqref{delta-recur} and the channel state. Hence, (i) is true.
    
    Next, (ii) follows since the instantaneous cost~\eqref{new-instan-cost} is a function of $(\D,c)$ and $u$.
    Thus, from~\cite[Ch. 6]{kumar2015stochastic} we have that $(\D(t), c(t))$ is an information state, and the optimization problem~\eqref{obj-fin-hor-w/o-log},\eqref{new-instan-cost} is a MDP with state $(\D(t),c(t)) \in \bR \times \{0,1\}$ and control $u(t) \in \{0,1\}$. This proves the Lemma.
\end{proof}

We now describe the controlled transition probabilities associated with~\eqref{obj-fin-hor-w/o-log_1}. Let $p(\D_+,c_+ \mid \D,c;u)$ denote the transition density function from the current state $(\D,c)$ to the next state $(\D_+,c_+)$ under action $u$. Consider the following two possibilities for $u$:

%\rs{any changes due to $p_{11}$?}

(i) $u=0:$ Then the corresponding transition density is,
\al{
& p(\D_+,c_+ \mid \D,c;0) \notag\\
& = p_{c0} e^{-\frac{(\D_+-a\D)^2}{2\sigma^2}} \delta_{0}(c_+) + p_{c1} e^{-\frac{(\D_+-a\D)^2}{2\sigma^2}} \delta_{1}(c_+).\label{df-0}
}

(ii) $u=1:$ Then the corresponding transition density is,
\al{
& p(\D_+,c_+ \mid \D,c;1) \notag\\
& = c\left[p_{c0} e^{-\frac{\D_+^2}{2\sigma^2}} \delta_{0}(c_+) + p_{c1} e^{-\frac{\D_+^2}{2\sigma^2}} \delta_{1}(c_+)\right] + (1-c)\notag\\
& \times\left[p_{c0} e^{-\frac{(\D_+-a\D)^2}{2\sigma^2}} \delta_{0}(c_+) + p_{c1} e^{-\frac{(\D_+-a\D)^2}{2\sigma^2}} \delta_{1}(c_+)\right],\label{df-1}
}

\subsection{Value Iteration}\label{val-iter}
We now show that we can use value iteration algorithm to solve~\eqref{obj-fin-hor-w/o-log_1}. Since we are dealing with a risk-sensitive cost objective, we firstly need to verify whether our MDP satisfies certain conditions~\cite[pp. 107-108]{bauerle2014more}. This is done next. We start with some definitions.
\begin{definition}[Transistion law]
    Let $\cL$ denote the Lebesgue measure on $\bR$. The controlled transition law denoted by $\{P(\cdot \mid \D, c, u)\}$ describes the transition probabilities for each $(\D,c,u) \in \bR \times \{0,1\}, \times \{0,1\}$, and has a density $p(\cdot \mid \D, c; u)$~\eqref{df-0}-\eqref{df-1} with respect to $\cL$~\cite[Example C.6]{hernandez2012discrete}, i.e., for any Borel measurable subset $\cB$ of $\bR$,
    \al{
        & P((\D_+,c_+) \in \cB \times \{0,1\} \mid \D, c, u)\notag\\
        & =\sum_{c_+ \in \{0,1\}} \int_{\cB} p(\D_+,c_+ \mid \D, c; u) d \cL(\D_+).\label{trans-law}
        }
\end{definition}\vspace{1em}

\begin{definition}[Weakly and strongly continuous]
	The transition law $\{P(\cdot \mid \D, c, u)\}$ is said to be
    
    (i) \emph{weakly continuous}, if for each $(\D,c,u) \in \bR \times \{0,1\}, \times \{0,1\}$, and continuous and bounded function $w: \bR \times \{0,1\} \rightarrow \bR$, the function $w':\bR \times \{0,1\} \times \{0,1\} \rightarrow \bR$ is continuous, where $w'(\D,c,u)= \bE[w \mid \D,c,u]$,

    (ii) \emph{strongly continuous}, if for each $(\D,c,u) \in \bR \times \{0,1\}, \times \{0,1\}$, and measurable bounded function $w: \bR \times \{0,1\} \rightarrow \bR$, the function $w':\bR \times \{0,1\} \times \{0,1\} \rightarrow \bR$ is continuous and bounded, where $w'(\D,c,u)= \bE[w \mid \D,c,u]$.
\end{definition}
%with state $(\D,c) \in \bR \times \{0,1\}$ and action $u \in \{0,1\}$ . 
\begin{lemma}\label{lemma:VI-conditions}
    MDP~\eqref{obj-fin-hor-w/o-log_1} satisfies the following properties:
    \begin{itemize}
        \item[{P1)}] The risk-sensitive criterion is continuous and strictly increasing in $\bR_+$.
        \item[{P2)}]  The action set is compact for all $(\D,c) \in \bR \times \{0,1\}$.
        \item[{P3)}] The function $(\D,c) \mapsto u$ is upper semicontinuous\footnote{A function $v: \bR \times \{0,1\} \rightarrow u$ is upper semicontinuous if its superlevel sets $\{(\D,c) \in \bR \times \{0,1\} \mid v(\D,c) \geq u'\}$ with $u' \in \{0,1\}$ are closed in $\bR \times \{0,1\}$.}.
        \item[{P4)}] The instantaneous cost is such that $(\D,c,u) \mapsto d(\D,c,u)$ is lower semicontinuous\footnote{A function $v$ is lower semicontinuous if $-v$ is upper semicontinuous.}.
        \item[{P5)}] The transition law $\{P(\cdot \mid \D, c, u)\}$ is weakly continuous for each $(\D,c,u) \in \bR \times \{0,1\}$. 
    \end{itemize}
\end{lemma}
\begin{proof}
    {P1)} follows since we have an exponential risk criterion.\par
        {P2)} and {P3)} follow since the action set is finite in our case, i.e., for every state $(\D,c) \in \bR \times \{0,1\}, u \in \{0,1\}$.\par
        {P4)} The instantaneous cost $d$~\eqref{new-instan-cost} is continuous in $\bR$ and hence, lower semicontinuous.\par
        {P5)} We show that $P$ is strongly continuous. The result then follows because strong continuity implies weak continuity~\cite[Definition C.3]{hernandez2012discrete}. We have for any Borel measurable subset $\cB$ of $\bR$,
        \al{
        & P((\D_+,c_+) \in \cB \times \{0,1\} \mid \D, c, u)\notag\\
        & =\sum_{c_+ \in \{0,1\}} \int_{\cB} p(\D_+,c_+ \mid \D, c; u)\, d\D_+, \notag
        }
        where first equality follows from~\eqref{trans-law} and because $d\mu(\D_+) = d\D_+$.\par
        Then, $P$ is strongly continuous from the definition of $p$~\eqref{df-0}-\eqref{df-1}~\cite[Example C.6]{hernandez2012discrete}. This completes the proof.
\end{proof}
The above result allows us to use the value iteration algorithm to solve~\eqref{obj-fin-hor-w/o-log_1}. For $(\D,c) \in \bR \times \{0,1\}$, define, 
\al{\label{valfunc-fin-hor}
& V(\D,c) := \min_{\phi} J_T(\D,c;\phi),
}
where for a policy $\phi$ we define,
\al{
J_T(\D,c;\phi) := \bE_{\phi}[e^{\gamma \sum_{t=0}^{T} d(\D(t),c(t),u(t))}].
}
Let $V_t$ be the iterate at stage $t$ of the value iteration algorithm. The next result follows from~\cite[Theorem 1, Corollary 1]{bauerle2014more} upon letting $U(y):=e^{\gamma y}$. It describes the value iteration algorithm for obtaining $V$, and also yields an optimal policy $\phi\ust$. 
\begin{proposition}\label{prop:valit-fin-hor}
    Consider the MDP~\eqref{obj-fin-hor-w/o-log_1} with transition density function $p$~\eqref{df-0}-\eqref{df-1}. Then,
    \begin{itemize}
        \item[a)] The iterates $V_t, t = 0,1,\ldots, T$ associated with the value iteration algorithm are generated as follows: for each $(\D, c) \in \bR \times \{0,1\}$, we have,
        \al{\label{Vn-fin-hor}
        V_{t+1}(\D,c) = \min_{u \in \{0,1\}} Q_{t+1}(\D,c ; u),
        }
        where for $u=0$,
        \al{
        & Q_{t+1}(\D,c;0) = e^{\g \D^2} \sum_{c_+ \in \{0,1\}} p_{cc_+}\notag\\
        & \times \int_{\bR} e^{-\frac{(\D_+ -a\D)^2}{2\sigma^2}} V_t(\D_+,c_+)\, d\D_+,\label{Qn0-fin-hor}
        }
        and for $u=1$,
        \al{
        & Q_{t+1}(\D,c;1) = (1-c) e^{\g(\lambda + \D^2)} \sum_{c_+ \in \{0,1\}} p_{cc_+}\notag\\
        & \times  \int_{\bR} e^{-\frac{(\D_+ -a\D)^2}{2\sigma^2}} V_t(\D_+,c_+)\, d\D_+ + c e^{\g \lambda}\notag\\
        & \times \sum_{c_+ \in \{0,1\}} p_{cc_+} \int_{\bR} e^{-\frac{\D_+^2}{2\sigma^2}} V_t(\D_+,c_+)\, d\D_+,\label{Qn1-fin-hor}
        }
        where,
        \al{
        V_0(\D,c) := 1. \label{v0-fin-hor}
        }
        \item[b)] There exists an optimal deterministic Markov policy $\phi\ust= (\phi_0\ust,\phi_1\ust,\ldots, \phi_T\ust)$, i.e., it chooses $u(t)$ only on the basis of $(\Delta(t),c(t))$, where for each $t=1,\ldots, T, \phi_n\ust(\D,c)$ attains the minimum in~\eqref{Vn-fin-hor} for each $(\D,c) \in \bR \times \{0,1\}$. Moreover, $V (\D,c)= V_T(\D,c)$, and $V (\D,c) =J_T(\D,c;\phi\ust)$ for each $(\D,c) \in \bR \times \{0,1\}$.
    \end{itemize}
\end{proposition}
\subsection{Folding the MDP}\label{folded MDP}
We now construct a certain folded MDP~\cite{chakravorty2018sufficient} by modifying the transition density function~\eqref{df-0}-\eqref{df-1} of the original MDP~\eqref{obj-fin-hor-w/o-log_1}. The state-space of the folded MDP is $\bR_+ \times \{0,1\}$, in contrast to the original MDP that has a state space $\bR \times \{0,1\}$. This folded MDP is much simpler to analyze than the original MDP. Specifically, the error in the folded MDP assumes only non-negative values, while in the original MDP the error takes both negative and non-negative values. It is shown in Proposition~\ref{prop:equiv} that the folded MDP is equivalent to the original MDP. Thus, the value function of the folded MDP agrees with that of the original function on its state-space, and one can recover an optimal policy for the original MDP by solving the folded MDP.~Hence, it suffices to work with the folded MDP for further analysis. We now derive a key property of the value iterates, $V_t, t \in \{0,1,\ldots,T\}$ of the original MDP~\eqref{obj-fin-hor-w/o-log_1} that is instrumental in constructing the folded MDP.
\begin{proposition}\label{prop:even-fin-hor}
    The functions $V_t(\cdot,c), Q_t(\cdot,c;u),~c\in \{0,1\}, u\in \{0,1\}$, $t \in \{0,1,\ldots, T\}$ are even, i.e., %we have the following for each $(\D,c) \in \bR \times \{0,1\}, u \in \{0,1\}$ and , 
    \nal{
    Q_t(\D,c;u)  =Q_t(|\D|,c;u), V_t(\D,c)=V_t(|\D|,c),
    }
    where $\Delta\in\bR$.
    Thus, if $\phi\ust(\cdot,c)$ is optimal, then we have,
    \nal{
    \phi_t\ust(\D,c)  =\phi_t\ust(|\D|,c).\notag%\label{even-fin-hor}
    }
\end{proposition}
\begin{proof}
        We prove this using induction. Since from~\eqref{v0-fin-hor} we have $V_0(\D,c)=1$ for $(\D,c) \in \bR \times \{0,1\}$, $V_0(\cdot,c)$ is even. This is the base case for induction. Next, assume that the iterates $V_s(\cdot,c), c \in \{0,1\}$, $s=1,2,\ldots, t$, are even. We will show that the functions $Q_{t+1}(\cdot,c; u), c \in \{0,1\}, u \in \{0,1\}$ are even. Consider the following two cases:\\
    Case i): $u=0.$ We have,
    \al{
    & Q_{t+1}(-\D,c;0) \notag\\
    & = e^{\g (-\D)^2}\sum_{c_+ \in \{0,1\}} p_{cc_+} \int_{\bR} e^{-\frac{(\D_+ +a\D)^2}{2\sigma^2}} V_t(\D_+,c_+)\, d\D_+ \notag\\
    & = e^{\g \D^2} \sum_{c_+ \in \{0,1\}} p_{cc_+} \int_{\bR} e^{-\frac{(-\D' +a\D)^2}{2\sigma^2}} V_t(-\D',c_+)\, d\D' \notag\\
    & = e^{\g \D^2} \sum_{c_+ \in \{0,1\}} p_{cc_+} \int_{\bR} e^{-\frac{(\D' -a\D)^2}{2\sigma^2}} V_t(\D',c_+)\, d\D' \notag\\
    & = Q_{t+1}(\D,c;0),\notag
    }
    where the first equality follows from~\eqref{Qn0-fin-hor}, the second equality follows from a change of variable by replacing $\D_+$ with $-\D'$, and finally the third equality follows from the induction hypothesis that $V_t(\cdot,c)$ is even. Hence, $Q_{t+1}(\cdot,c;0)$ is even.
    
    Case ii): $u=1.$ We have,
    \al{
    & Q_{t+1}(-\D,c;1) = (1-c) e^{\g(\lambda + \D^2)} \notag\\
        & \times \sum_{c_+ \in \{0,1\}} p_{cc_+} \int_{\bR} e^{-\frac{(\D_+ +a\D)^2}{2\sigma^2}} V_t(\D_+,c_+)\, d\D_+\notag\\
        & + c e^{\g \lambda} \sum_{c_+ \in \{0,1\}} p_{cc_+} \int_{\bR} e^{-\frac{\D_+^2}{2\sigma^2}} V_t(\D_+,c_+)\, d\D_+\notag\\
    & = (1-c) e^{\g(\lambda + \D^2)} \notag\\
        & \times \sum_{c_+ \in \{0,1\}} p_{cc_+} \int_{\bR} e^{-\frac{(-\D' +a\D)^2}{2\sigma^2}} V_t(-\D',c_+)\, d\D'\notag\\
        & + c e^{\g \lambda} \sum_{c_+ \in \{0,1\}} p_{cc_+} \int_{\bR} e^{-\frac{\D'^2}{2\sigma^2}} V_t(-\D,c_+)\, d\D'\notag\\
    & =Q_{t+1}(\D,c;1),\notag
    }
    where the first equality follows from~\eqref{Qn1-fin-hor}, the second equality follows from a change of variables, while the third equality follows from our induction hypothesis that $V_t(\cdot,c)$ is even and~\eqref{Qn1-fin-hor}. This shows that $Q_{t+1}(\cdot,c;1)$ is also even. Now, $V_{t+1}(\cdot,c)$ is even  since from~\eqref{Vn-fin-hor} we have that $V_{t+1}(\cdot,c)$ is pointwise minimum of two even functions $Q_t(\cdot,c;0)$ and $Q_t(\cdot,c;1)$.~Since from Proposition~\ref{prop:valit-fin-hor} b) we have that $\phi_{t+1}\ust(\cdot,c) \in \argmin_{u \in\{0,1\}} Q_{t+1}(\cdot,c;u)$, $\phi\ust_{t+1}(\cdot,c)$ is also even. The claim then follows from induction.        
\end{proof}
We next construct the folded MDP~\cite{chakravorty2018sufficient}. We use $\TD, \Tc, \Tu$ and $\Tph$ to denote the error, channel state, action, and policy, respectively for the folded MDP.
\begin{definition}[Folded MDP]
    Consider the MDP~\eqref{obj-fin-hor-w/o-log_1} that has a transition density function $p$~\eqref{df-0}-\eqref{df-1}. The associated folded MDP is a MDP with state-space $\bR_+ \times \{0,1\}$, control space $\{0,1\}$, and transition density function $\Tp$ given as follows,
    \al{\label{df-fold}
    & \Tp(\TD_+,\Tc_+ \mid \TD, \Tc; \Tu) \notag\\
    & = p(\TD_+,\Tc_+ \mid \TD, \Tc; \Tu) + p(-\TD_+,\Tc_+ \mid \TD, \Tc; \Tu),
    }
    where $\TD, \TD_+ \in \bR_+, \Tc, \Tc_+ \in \{0,1\}$, and $\Tu \in \{0,1\}$. The objective function~\eqref{obj-fin-hor-w/o-log_1} and the instantaneous cost $\Td$~\eqref{new-instan-cost} remain the same. 
\end{definition}
Define,
\al{ \psi(v):=e^{-\frac{v^2}{2\sigma^2}} , \varphi(v,s):=\psi(v-s) + \psi(v+s).\notag}
Next, we can show that the properties P1)-P5) stated in Lemma~\ref{lemma:VI-conditions} are satisfied by the folded MDP too. The proof is similar to Lemma~\ref{lemma:VI-conditions}. Hence, we can use value iteration to solve the folded MDP, and there exists an optimal deterministic Markov policy. These results follows from~\cite{bauerle2014more}, and are analogous to Proposition~\ref{prop:valit-fin-hor}, that was shown for the original MDP~\eqref{obj-fin-hor-w/o-log_1}. Let $\TV_t$ denote the iterate at stage $t$ when the value iteration algorithm is used to solve the folded MDP. Then, for $(\TD,\Tc) \in \bR_+ \times \{0,1\}$, and $t=0,1,\ldots, T-1$, we have,
\al{\label{Vn-fold}
\TV_{t+1}(\TD,\Tc) = \min_{u \in \{0,1\}} \TQ_{t+1}(\TD,\Tc;u),
}
where,
\al{
        & \TQ_{t+1}(\TD,\Tc;0) \notag\\
        & = e^{\g \TD^2} \sum_{\Tc_+ \in \{0,1\}} \int_{\bR_+} \Tp(\TD_+,\Tc_+ \mid \TD, \Tc;0) \TV_{t}(\TD_+,\Tc_+)\, d\TD_+\notag\\
        & = e^{\g \TD^2} \notag\\
        & \times \sum_{\Tc_+ \in \{0,1\}} p_{\Tc \Tc_+} \int_{\bR_+} \varphi(\TD_+,a\TD) \TV_t(\TD_+,\Tc_+)\, d\TD_+\label{Qn0-fold}\\
        & \TQ_{t+1}(\TD,\Tc;1) = (1-\Tc) e^{\g(\lambda + \TD^2)} \notag\\
        & \times \sum_{\Tc_+ \in \{0,1\}} p_{\Tc\Tc_+} \int_{\bR_+} \varphi(\TD_+,a\TD)  V_t(\TD_+,\Tc_+)\, d\TD_+\notag\\
        & + 2\Tc e^{\g \lambda} \sum_{\Tc_+ \in \{0,1\}} p_{\Tc\Tc_+} \int_{\bR_+} \psi(\TD_+) V_t(\TD_+,\Tc_+)\, d\TD_+,\label{Qn1-fold}
        }
        where,
        \al{
        \TV_0(\TD,\Tc) := 1, \label{v0-fold}
    }
and where~\eqref{Qn0-fold} and~\eqref{Qn1-fold} follow from the definition of $\Tp$~\eqref{df-fold}.

We now prove the equivalence of the folded MDP with state-space $\bR_+ \times \{0,1\}$ with the original MDP~\eqref{obj-fin-hor-w/o-log_1} with state-space $\bR \times \{0,1\}$ in the following Proposition. This allows us to use the folded MDP for subsequent analysis. We use $\Tph\ust=(\Tph_0\ust, \Tph_1\ust,\ldots,\Tph_T\ust)$ to denote an optimal deterministic Markov policy for the folded MDP.
\begin{proposition}\label{prop:equiv}
    The functions $\TQ_t, \TV_t$ corresponding to the folded MDP agree with $Q_t, V_t$~\eqref{Vn-fin-hor}-\eqref{Qn1-fin-hor} of the original MDP on $\bR_+ \times \{0,1\}$, i.e., we have the following for each $(\D,c) \in \bR \times \{0,1\}, u \in \{0,1\}, t \in \{0,1,\ldots, T\},$
    \al{
    & Q_t(\D,c;u) = \TQ_t(|\D|,c;u), V_t(\D,c)=\TV_t(|\D|,c).\notag
    }
    Thus, for any optimal policy $\Tph\ust$ (folded MDP),~$\phi\ust$ (original MDP),  
    \al{
    & \phi_t\ust(\D,c) = \Tph_t\ust(|\D|,c).\notag
    }
\end{proposition}
\begin{proof}
    We will prove the claim via induction. %From Proposition~\ref{prop:even-fin-hor}, the functions $V_0(\cdot,c), c \in \{0,1\}$ are even. Hence, we have 
    Note that from ~\eqref{v0-fin-hor},~\eqref{v0-fold} we have $V_0(\D,c) = V_0(|\D|,c)= 1$ and also $\TV_0(\D,c) =1$. This is the base case for induction. Next, assume that for each $(\D,c) \in \bR_+ \times \{0,1\}, V_s(\D,c)=\TV_s(\D,c)$ for $s=1,2,\ldots,t$. We will now show that, $Q_{t+1}(\D,c;u)=\TQ_{t+1}(\D,c;u)$. For this purpose, consider the following two cases for each $(\D,c) \in \bR_+ \times \{0,1\}$:\par
    Case i): $u=0.$ We have,
    \al{
    & Q_{t+1}(\D,c;0) \notag\\
    & = e^{\g \D^2} \sum_{c_+ \in \{0,1\}} p_{cc_+} \int_{\bR} \psi(\D_+-a\D) V_t(\D_+,c_+)\, d\D_+\notag\\
    & = e^{\g \D^2} \sum_{c_+ \in \{0,1\}} p_{cc_+} \biggl[\int_{\bR_+} \psi(\D_+-a\D) V_t(\D_+,c_+)\, d\D_+ \biggr.\notag\\
    & \biggl. + \int_{\bR_-} \psi(\D_+-a\D) V_t(\D_+,c_+)\, d\D_+\biggl]\notag\\
    & = e^{\g \D^2} \sum_{c_+ \in \{0,1\}} p_{cc_+} \biggl[\int_{\bR_+} \psi(\D_+-a\D) V_t(\D_+,c_+)\, d\D_+ \biggr.\notag\\
    & \biggl. + \int_{\bR_+} \psi(-\D_+-a\D) V_t(-\D_+,c_+)\, d\D_+\biggl]\notag\\
    & = e^{\g \D^2} \sum_{c_+ \in \{0,1\}} p_{cc_+} \int_{\bR_+} \varphi(\D_+,a\D) \TV_t(\D_+,c_+)\, d\D_+\notag\\
    & = \TQ_{t+1}(\D,c;0),\label{eq:1}
    }
    where the first equality follows from~\eqref{Qn0-fin-hor}. The third equality follows from Proposition~\ref{prop:even-fin-hor}, and the induction hypothesis that $V_t(\D,c)=\TV_t(\D,c)$. Finally, the last equality follows from~\eqref{Qn0-fold}.
    
    Case ii): $u=1.$ We have,
    \al{
    & Q_{t+1}(\D,c;1) =\cQ\notag\\
    & \times \sumc p_{cc_+} \int_{\bR} \psi(\D_+-a\D) V_t(\D_+,c_+)\, d\D_+\notag\\
    & + \gl \sumc p_{cc_+} \int_{\bR} \psi(\D_+) V_t(\D_+,c_+)\, d\D_+\notag\\
    & = \cQ \sumc p_{cc_+}\notag\\
    & \times \biggl[\itg \psi(\D_+-a\D) V_t(\D_+,c_+)\, d\D_+ \biggr.\notag\\
    & \biggl. + \itg \psi(-\D_+-a\D) V_t(-\D_+,c_+)\, d\D_+ \biggr]\notag\\
    & + \gl \sumc p_{cc_+} \biggl[\itg \psi(\D_+) V_t(\D_+,c_+)\, d\D_+ \biggr.\notag\\
    & + \biggl. \itg \psi(\D_+) V_t(-\D_+,c_+)\, d\D_+ \biggr]\notag\\
    & = \TQ_{t+1}(\D,c;1),\label{eq:2}
    }
    where the first equality follows from~\eqref{Qn1-fin-hor} and the last equality follows from Proposition~\ref{prop:even-fin-hor}, our induction hypothesis that $V_t(\D,c)=\TV_t(\D,c)$, and~\eqref{Qn1-fold}.
    
    Now, upon combining~\eqref{eq:1},~\eqref{eq:2} with Proposition~\ref{prop:even-fin-hor}, we obtain $Q_{t+1}(\D,c;u)= Q_{t+1}(|\D|,c;u)=\TQ_{t+1}(|\D|,c;u)$ for each $(\D,c) \in \bR \times \{0,1\}$. Next, from~\eqref{Vn-fin-hor},~\eqref{Vn-fold} we have that $V_{t+1}, \TV_{t+1}$ is the pointwise minimum of $Q_{t+1},\TQ_{t+1}$ taken with respect to $u \in \{0,1\}$, we have that $V_{t+1}(\D,c)=\TV_{t+1}(|\D|,c)$. Since $\phi_{t+1}\ust$ 
    chooses the action that minimizes the function $Q_{t+1}(\Delta,c;\cdot)$, similarly $\Tph_{t+1}\ust$ chooses action which minimizes $\TQ_{t+1}(\D,c;\cdot)$, we have $\phi_{t+1}\ust(\D,c) = \Tph_{t+1}\ust(|\D|,c)$ for each $(\D,c) \in \bR \times \{0,1\}$. The claim then follows from induction.
\end{proof}

%% file: Structural_results.tex
\section{Structural results}\label{sec:structural results}
In this section, we begin by showing some structural results for optimal policy of the folded MDP $(\bR_+ \times \{0,1\}, \{0,1\}, \Tp, \Td)$. Specifically, we first establish in Proposition~\ref{prop:monotoneV} a result on the monotonicity property of the value function iterates $\TV_t, t \in \{0,1,\ldots,T\}$. Next, we show that an optimal scheduling policy $\Tph\ust$ satisfies a certain structure. Finally, by using Proposition~\ref{prop:equiv}, we obtain similar structural results for the original MDP~\eqref{obj-fin-hor-w/o-log_1} by unfolding this MDP.
%We now show that the value iterates satisfies the following property.
\begin{proposition}\label{prop:monotoneV}
    Consider the folded MDP $(\bR_+ \times \{0,1\}, \{0,1\}, \Tp, \Td)$. For each $\Tc \in \{0,1\}$, the iterates generated by the value iteration algorithm~\eqref{Vn-fold} $\TV_t(\cdot,\Tc), t \in 0,1,\ldots,T$ are non-decreasing (with respect to $\TD$).
\end{proposition}
\begin{proof}
    We will prove this via induction. Since $\TV_0(\TD,\Tc)=1$~\eqref{v0-fold}, the claim holds for $n=0$. Next, assume that the functions $\TV_s(\cdot,\Tc), \Tc \in \{0,1\}, s=1,2,\ldots,t$ are non-decreasing. We will now show that the functions $\TQ_{t+1}(\cdot,\Tc;\Tu), \Tc \in \{0,1\}, \Tu \in \{0,1\}$ are non-decreasing. For this purpose, consider $\TD',\TD \in \bR_+$ satisfying
    $\TD' \geq \TD$. We have the following two possibilities:

	 Case i): $\Tu=0.$ We have,
	\al{
		& \TQ_{t+1}(\TD',\Tc;0) \notag\\
		& = e^{\g \TD'^2} \sum_{\Tc_+ \in \{0,1\}} p_{\Tc \Tc_+} \itg \varphi(\TD_+,a\TD')\TV_t(\TD_+,\Tc_+)\, d\TD_+\notag\\
		& \geq \gTd \sumTc p_{\Tc\Tc_+} \itg \vph \TV_t(\TD_+,\Tc_+)\, d\TD_+\notag\\
		& = \TQ_{t+1}(\TD,\Tc;0),\label{eq:3}
	}
	where the first equality follows from~\eqref{Qn0-fold}, and the inequality follows from our induction hypothesis on $\TV_t$ and Lemma~\ref{appen:lemma-inc-func}.
 
	 Case ii): $\Tu=1.$ We have,
	\al{
		& \TQ_{t+1}(\TD',\Tc;1) = (1-\Tc) e^{\g(\lambda + \TD'^2)} \notag\\
		& \times \sum_{\Tc_+ \in \{0,1\}} p_{\Tc\Tc_+} \int_{\bR_+} \varphi(\TD_+,a\TD')  V_t(\TD_+,\Tc_+)\, d\TD_+\notag\\
		& + 2\Tc e^{\g \lambda} \sum_{\Tc_+ \in \{0,1\}} p_{\Tc\Tc_+} \int_{\bR_+} \psi(\TD_+) V_t(\TD_+,\Tc_+)\, d\TD_+\notag\\
		& \geq (1-\Tc) e^{\g(\lambda + \TD^2)} \notag\\
		& \times \sum_{\Tc_+ \in \{0,1\}} p_{\Tc\Tc_+} \int_{\bR_+} \varphi(\TD_+,a\TD)  V_t(\TD_+,\Tc_+)\, d\TD_+\notag\\
		& + 2\Tc e^{\g \lambda} \sum_{\Tc_+ \in \{0,1\}} p_{\Tc\Tc_+} \int_{\bR_+} \psi(\TD_+) V_t(\TD_+,\Tc_+)\, d\TD_+\notag\\
		& = \TQ_{t+1}(\TD,\Tc;1),\label{eq:4}
	}
	where the first equality follows from~\eqref{Qn1-fold}, and the inequality follows from induction hypothesis on $\TV_t$ and Lemma~\ref{appen:lemma-inc-func} in the appendix.
	
	Since $\TV_{t+1}$ is the pointwise minimum of $\TQ_{t+1}$ taken w.r.t. $\Tu \in \{0,1\}$, from~\eqref{eq:3} and~\eqref{eq:4} we have that $\TV_{t+1}(\cdot,\Tc)$ is non-decreasing.~The proof then follows from induction.
\end{proof}    
We now introduce the class of threshold-type policies for the folded MDP, and for the original MDP. We will then show that an optimal scheduling policy for the folded MDP belongs to this class.
\begin{definition}[Threshold-type Policy]
    Let the channel state and error at time $t \in \{0,1,\ldots,T\}$ for the folded MDP be $\Tc$ and $\TD$ respectively. We say that a scheduling policy $\Tph$ for the folded MDP is of threshold-type if for each $t \in \{0,1,\ldots,T\}$ and $\Tc \in \{0,1\}$ there exists a threshold $\TD\ust_t(\Tc)$ such that it attempts packet transmission at time $t$ only when $\TD \geq \TD\ust_t(\Tc)$. 
    
    Similarly, a scheduling policy $\phi$ of the original MDP is of threshold-type if for each $c \in \{0,1\},$ there exists a threshold $\D\ust_t(c)$ such that a transmission attempt at time $t$ occurs only when the error $\Delta$ exceeds the corresponding threshold, i.e. when $|\D| \geq \D\ust_t(c)$.
\end{definition}
The following theorem shows that the optimal scheduling policy for the folded MDP exhibits a threshold structure.
\begin{theorem}\label{main-theorem}
    Let $\Tc$ and $\TD$ be the channel state and error at time $t \in \{0,1\ldots,T\}$ respectively. Then, for each $t$ and $ \Tc \in \{0,1\}$, there exists a threshold $\TD\ust_t(\Tc)$ such that it is optimal to transmit at time $t$ only when $\TD \geq \TD\ust_t(\Tc)$. Thus, there exists an optimal scheduling policy that admits a threshold structure.
\end{theorem}
\begin{proof}
    We will first show that for each time $t \in \{0,1,\ldots, T\}$, it is optimal to not transmit when the channel state is bad $(\Tc(t) = 0)$. Hence, scheduler only has to choose between the actions $0$ and $1$ when channel is good, i.e. when $\Tc(t) = 1$. For this purpose, consider the following two cases:
    
    Case i): $\Tc =0.$ We have,
    \al{
    & \TQ_{t}(\TD,0;1) \notag\\
    & = e^{\g (\lambda + \TD^2)} \sumTc p_{\Tc \Tc_+} \itg \vph V_{t-1}(\Td_+,\Tc_+)\notag\\
    & \geq \gd \sumTc p_{\Tc \Tc_+} \itg \vph V_{t-1}(\Td_+,\Tc_+)\notag\\
    & = \TQ_{t}(\TD,0;0),\notag
    }
    where the first equality follows from~\eqref{Qn1-fold}, and the inequality follows since $\g, \lambda > 0$. Hence, $\Tph_t\ust(\TD,0)=0$ for each 
    $t$. Since this is a trivial threshold policy, the claim holds for $\Tc=0$.\par
    Case ii): $\Tc=1.$ In this case, showing threshold structure is equivalent to showing that if $\TQ_t(\TD,1;1) \leq \TQ_t(\TD,1;0)$, then $\TQ_t(\TD',1;1) \leq \TQ_t(\TD',1;0)$ for $\TD' \geq \TD$. So consider,
    \al{
    & \TQ_t(\TD',1;1) - \TQ_t(\TD',1;0)\notag\\
    & = 2\Tc e^{\g \lambda} \sum_{\Tc_+ \in \{0,1\}} p_{\Tc\Tc_+} \int_{\bR_+} \psi(\TD_+) V_t(\TD_+,\Tc_+)\, d\TD_+\notag\\
    & - e^{\g \TD'^2} \sum_{\Tc_+ \in \{0,1\}} p_{\Tc \Tc_+} \itg \varphi(\TD_+,a\TD')\TV_t(\TD_+,\Tc_+)\, d\TD_+\notag\\
    & \leq 2\Tc e^{\g \lambda} \sum_{\Tc_+ \in \{0,1\}} p_{\Tc\Tc_+} \int_{\bR_+} \psi(\TD_+) V_t(\TD_+,\Tc_+)\, d\TD_+\notag\\
    & - e^{\g \TD^2} \sum_{\Tc_+ \in \{0,1\}} p_{\Tc \Tc_+} \itg \varphi(\TD_+,a\TD)\TV_t(\TD_+,\Tc_+)\, d\TD_+\notag\\
    & = \TQ_t(\TD,1;1) - \TQ_t(\TD,1;0)\notag\\
    & \leq 0,\notag
    }
    where the inequality follows from Proposition~\ref{prop:monotoneV} and Lemma~\ref{appen:lemma-inc-func} in the appendix. This completes the proof.
\end{proof}
We will now unfold the folded MDP $(\bR_+ \times \{0,1\}, \{0,1\}, \Tp, \Td)$ to get the original MDP~\eqref{obj-fin-hor-w/o-log_1}. As is shown next, this gives us a structural result for an optimal policy of the original MDP.
\begin{corollary}
     There exists an optimal scheduling policy $\phi\ust$ for the original MDP~\eqref{obj-fin-hor-w/o-log_1}, that exhibits a threshold structure.
\end{corollary}
\begin{proof}
    The result follows from Lemma~\ref{lemma:appen-incV} in the appendix, Proposition~\ref{prop:equiv}, and Theorem~\ref{main-theorem}.
\end{proof}

%% file: conclusions.tex
\section{Conclusion}
In this work, we consider a remote state estimation setup in which a sensor observes an AR Markov process and has to dynamically decide whether or not to transmit an update to the remote estimator via an unreliable wireless channel that is modeled as a Gilbert-Elliott channel. The objective is to minimize a risk-sensitive cost criterion which is the expected value of the exponential of the cumulative costs incurred over a finite time horizon. The instantaneous costs are the sum of the power consumption, and estimation error. Due to the consideration of risk sensitive objective, the procedure also penalizes higher-order moments of the cumulative cost in addition to its mean. We formulate this optimization problem as a MDP. Since the original MDP MDP was difficult to aanalyze, to facilitate the analysis, we constructed a folded MDP and showed that it is equivalent to the original MDP. Subsequently, we developed an optimal policy for the folded MDP and showed that it has a threshold structure, i.e., the sensor transmits a packet only when the current error exceeds a certain threshold. Upon unfolding this MDP, we obtained similar structural results for the original problem. This work can be extended in several interesting directions. Firstly, we would like to jointly optimize over the choice of estimator and scheduler.
%explore the joint optimization of both the estimator and the scheduler, rather than fixing the estimator. 
Secondly, we aim to extend these results to an infinite horizon setup. Moreover, since the state space is infinite, this renders the use of value iteration algorithm impractical. We would like to develop a computationally efficient algorithm that approximates the optimal policy well. For the infinite horizon setup, we would also like to develop stationary policies that are optimal.~Finally, we assumed that the system parameter and channel parameters are known. Since this knowledge is difficult to obtain in practice, we would like to derive an efficient learning algorithm that would learn a jointly optimal scheduler and estimator.

%% file: appendix.tex
\appendix

\section{Preliminary Lemma}\label{lemma:inc-func}
\renewcommand{\thesection}{\Alph{section}}
For ease of reference, we restate the notation here:
\al{ \psi(v):=e^{-\frac{v^2}{2\sigma^2}} , \varphi(v,s):=\psi(v-s) + \psi(v+s).\notag}
\begin{lemma}\label{lemma:appen-incV}
    Consider the original MDP~\eqref{obj-fin-hor-w/o-log_1}. For each $c \in \{0,1\}$, the value iterates $V_t(\D,c)$~\eqref{Vn-fin-hor}, $t \in \{0,1,\ldots,T\}$ are non-decreasing in $|\D|.$
\end{lemma}
\begin{proof}
    From Proposition~\ref{prop:equiv} we have that, $\D \in \bR_+, V_t(\D,c)=\TV_t(\D,c),$ . The result then follows from Proposition~\ref{prop:even-fin-hor} since for $\D \in \bR, V_t(\D,c)=V_t(|\D|,c)=\TV_t(|\D|,c)$ and $\TV_t(|\D|,c)$ in non-decreasing in $|\D|$ by Proposition~\ref{prop:monotoneV}.
\end{proof}
\begin{lemma}\label{appen:lemma-inc-func}
    Consider the folded MDP $(\bR_+ \times \{0,1\}, \{0,1\}, \Tp,\Td)$. Let $\TD' \geq \TD$ where $\TD',\TD \in \bR_+$. Assume for each $\Tc \in \{0,1\}, \TV_t(\cdot,\Tc), t \in \{0,1,\ldots,T\}$ is non-decreasing. Then, $\TV_t$ satisfies the following,
    \al{
    & \itg \varphi(\TD_+,a\TD')\TV_t(\TD_+,\Tc_+)\, d\TD_+ \notag\\
    & \geq \itg \varphi(\TD_+,a\TD)\TV_t(\TD_+,\Tc_+)\, d\TD_+\notag
    }
\end{lemma}
\begin{proof}
    The proof follows from~\cite{10384144} with $\varphi, \TD',\TD, \TD_+,\cT(\Tilde{b})$ replaced by $\psi, \Te',\Te,\Te_+,\Tc_+$ respectively.
\end{proof}